\documentclass[12pt,twoside]{article}
\usepackage[utf8]{inputenc}
\usepackage[english]{babel}

\usepackage{mathtools}
\usepackage{amsmath}
\usepackage{amsfonts,amssymb}
\usepackage{amsmath,amssymb,amsopn,amsthm}

\usepackage{graphicx}                 

\usepackage{cite}

\usepackage{graphicx}
\usepackage{epsfig, epsf}
\usepackage{epstopdf}
\usepackage{subfigure}
\usepackage{comment}

\theoremstyle{plain}
\newtheorem{lemma}{Lemma}
\newtheorem{theorem}{Theorem}
\newtheorem{definition}{Definition}
\newtheorem{remark}{Remark}

\newtheorem{lemma*}{Lemma}

\newcommand{\inter}{{\mbox{Int\ }}}

\newcommand{\vGamma}{{{\mathbb V}(\Gamma)}}
\newcommand{\eGamma}{{{\mathbb E}(\Gamma)}}
\newcommand{\eTwoGamma}{{{\mathbb E}^{(2)}(\Gamma)}}
\newcommand{\unit}{{\mathbf 1}}

\newcommand{\SByT}{{S \slash T}}
\newcommand{\TByS}{{T \slash S}}
\def\qed{\hfill$\square$\par\rm}

\begin{document}

\markboth
{Vladimir Tarkaev}
{A prime decomposition for string links in a thickened surface}


\title{
A prime decomposition theorem for string links in a thickened surface
}

\author{
Vladimir Tarkaev
\footnote{
Steklov Mathematical Institute of Russian Academy of Sciences, Moscow, Russia, 
Chelyabinsk State University, Chelyabinsk, Russia,
Krasovskii Institute of Mathematics and Mechanics, Ural Branch of the Russian Academy of Sciences, Yekaterinburg, Russia,
trk@csu.ru
}
}

\maketitle

\begin{abstract}
We prove a prime decomposition theorem for string links in a thickened surface.
Namely, we prove that any non-braid string link
$\ell \subset \Sigma \times I$,
where $\Sigma$ is a compact orientable (not necessarily closed) surface other than $S^2$,
can be written in the form
$\ell =\ell_1 \# \ldots \# \ell_m$,
where $\ell_j,j=1,\ldots,m,$ is prime string link defined up to braid equivalence,
and the decomposition is unique up to possibly permuting the order of factors in its right-hand side.
\end{abstract}

\section*{Introduction}
\label{sec:Introduction}

The concept of string links (including the one on surfaces)
was proposed by Milnor~\cite{Milnor54}
as a generalization of braids defined by Artin~\cite{artin47}.
String links are defined as braids with the monotonicity requirement relaxed,
i.e., they are a proper embedding of the disjoint union of $n \geq 1$ segments in the cylinder of the form
(compact orientable surface)$\times I$
so that the image of each segment has one endpoint on the bottom base of the cylinder and the other endpoint on its top base.
The subject appears and is still actively studied
in connection with the homotopy braid group,
that is the group formed by string links regarded up to so called link homotopy
(see, e.g.,~\cite{Milnor54},~\cite{HabeggerLin90}).
Another subject of investigation in this area is the monoid of string links
in the cylinder $D^2 \times I$
(see, e.g.,~\cite{Krebes},~\cite{MeilhanYasuhara14}),
in particular, it is closely related to Milnor invariants of classical links
(\cite{HabeggerLin90},~\cite{HabeggerLin98}).
The notion of string links in $D^2 \times I$ can be generalized in different directions.
For example, such string links
are just the same as classical long links,
hence it is natural to extend them to  long virtual links
(see, e.g.,~\cite{Cheng17},~\cite{Gaudreau20}).
Another natural generalization is string links in a thickened surface
(see, e.g.,~\cite{DuzhinKarev07},~\cite{HabeggerMeilhan08}),
and the latter kind of string links is what we consider in the present paper.

Our goal is to prove a prime decomposition theorem for string links on $n \geq 1$ strands
in $\Sigma \times I$, where $\Sigma$ is a compact orientable surface with (maybe) non-empty boundary.
Namely, in Theorem~\ref{theorem:Main} we prove that
a non-braid string link $\ell \subset \Sigma \times I$ can be written in the form
$\ell =\ell_1 \# \ldots \# \ell_m$,
where $\ell_j,j=1,\ldots,m,$ is prime string link defined up to braid equivalence,
and the decomposition is unique up to possibly permuting the order of factors in its right-hand side
(precise  definitions of all of these terms are given below in Section~\ref{sec:Preliminaries}).
This theorem is a generalization of an analogous result proved in~\cite{BlairBurkeKoytcheff}
for the particular case when $\Sigma$ is $D^2$ and $n=2$.
However, it is necessary to note that the authors in the aforementioned paper
described the center of the  monoid of $2$-string links in $D^2 \times I$.
With this result they give a precise answer to the question
on which factors in the decomposition can be permuted.
Also, we would like to mention a prime decomposition theorem for
long virtual knots~\cite{Chrisman13}.
This theorem covers the case of $1$-strand string links in (annulus)$\times I$.
The case is well studied because it is one of the possible interpretations
of knotoids proposed by Turaev~\cite{Turaev2012}.
In particular, Turaev in this work describes the center of the semigroup of knotoids.
In the general case, the structure of the center of the string link monoid remains unknown.

The paper is organized as follows.
Section~\ref{sec:Preliminaries}
gives all the necessary definitions and using notation.
In Section~\ref{sec:MainTheorem}
we formulate Theorem~\ref{theorem:Main},
which is our main result,
and then give a few comments concerning the theorem.
In Section~\ref{sec:DiamondLemma}
we briefly review the Diamond lemma,
which plays an essential role in our proof of Theorem~\ref{theorem:Main}.
Section~\ref{sec:ProofOfMainTheorem}
gives the proof of Theorem~\ref{theorem:Main}.

\section{Preliminaries}
\label{sec:Preliminaries}

\subsection{String links}
\label{sec:StringLinks}

Let $\Sigma$ be a compact connected orientable (not necessarily closed) surface other than $S^2$.
We denote by $M$ the manifold of the form
\[
M =\Sigma \times I,
\] 
where $I =[-1,1]$.
If $\partial \Sigma \not=\emptyset$ the structure of the direct product in $M$
is assumed to be fixed.
It is convenient to think that the surfaces $\Sigma \times \{ y \}, y \in I,$ lie in the cylinder $M$ horizontally,
while the segment $I$ determines a vertical coordinate increasing from bottom to top.

We denote by
\[
\partial_{-} =\Sigma \times \{ -1 \}, \partial_{+} =\Sigma \times \{ 1 \},
\]
i.e., $\partial_{-}$ and $\partial_{+}$ are the bottom and top bases
of the cylinder $M$, respectively.

Let $X \subset \inter \Sigma$
be a fixed set consisting of $n \geq 1$ pairwise distinct points lying in the interior of $\Sigma$.

\begin{definition}
[String link]
\label{def:StringLink}
By an \emph{$n$-component string link} ($n$-string link for short)
we understand a smooth embedding 
$I_1 \amalg \ldots \amalg I_n \to M \setminus (\partial \Sigma \times I)$
of $n$ pairwise disjoint copies of a segment into $M \setminus (\partial \Sigma \times [-1, 1])$
so that for any $j=1,\ldots,n,$ the image of $I_j$
has one endpoint in $X \times \{ -1 \}$ and the other endpoint in $X \times \{ 1 \}$.
The image of $I_j$ under the embedding is called the \emph{$j$th component} (or $j$th strand) of the string link.
\end{definition}
Some authors (e.g., see~\cite{BlairBurkeKoytcheff}) consider so-called pure string links,
those are string links with additional property:
for any $j=1,\ldots,n,$ the endpoints of the $j$th component
are $x_j \times \{ -1 \}$ and $x_j \times \{ 1 \}$ for some $x_j \in X$.
We do not require the property.

We will say that a string link is \emph{classical}
if $M =D^2 \times I$, where $D^2$ denotes the disk.

\begin{definition}
[String link equivalence]
\label{def:Equivalence}
String links $\ell_1,\ell_2 \subset M$ are \emph{equivalent}
if there is an isotopy fixing $\partial M$ that takes $\ell_1$ to $\ell_2$.
In this case,  (by abuse of notation) we will write $\ell_1 =\ell_2$.
\end{definition}

Let $p_I: M \to I$ be the projection map to the second coordinate in the direct product $\Sigma \times I$.
A string link $\ell$ is called a \emph{braid}
if there is a string link $\ell'$ equivalent to $\ell$ so that
the restriction of $p_I$ to every component  of $\ell'$ is a bijection.
This kind of braid is known in literature as surface braid.
Surface braids, like classical ones, form a group;
see, for example,
\cite{Scott70}, \cite{Gonzalez2001}, \cite{Bellingeri2004},
where the subject is studied from the algebraic point of view.
We will denote by $\unit$ the braid
that is equivalent to the string link
whose components are vertical segments $x_j \times I$.

The set of $n$-string links has a monoid structure,
given by the operation called the \emph{stacking product}.
The operation consists in identifying the top base of the cylinder containing $\ell_1$
with the bottom base of the cylinder containing $\ell_2$
and the vertical compression of the resulting cylinder to make its height standard.
 The stacking product operation is denoted by  $\ell_1 \# \ell_2$.

\begin{definition}
[Braid equivalence of string links]
\label{def:BraidEquivalence}
String links $\ell_1,\ell_2 \subset \Sigma \times I$ are \emph{braid equivalent}
if there is an isotopy fixing $\partial \Sigma \times I$
that takes $\ell_1$ to $\ell_2$.
In this case, we will write $\ell_1 \simeq \ell_2$.
\end{definition}
Since $B \simeq \unit$ for any braid $B$
we can rewrite the above definition in the form:
\[
\ell_1 \simeq \ell_2 \text{ if and only if } \ell_1 =B_1 \# \ell_2 \# B_2
\] 
for some braids $B_1,B_2$.
In the case of classical string links, the notion of braid equivalence via stacking products with braids
was proposed and studied in~\cite{BlairBurkeKoytcheff}.
In particular, the paper gives proof of the aforementioned fact.
The latter proof can be extended to our situation with minor obvious changes.

Note that in the case of classical $2$-component string links
$\ell_1 \# \ell_2 \simeq \ell'_1 \# \ell'_2$ if $\ell'_j \simeq \ell_j,j=1,2$.
However, in general, this fact is not true.

\subsection{Free string links}
\label{sec:FreeStringLink}

We have two different equivalence relations on the set of string links:
the first equivalence relation requires that the endpoints of strands are fixed
(see Definition~\ref{def:Equivalence})
while the second one allows us to move these endpoints
(see Definition~\ref{def:BraidEquivalence}).
The fixing of endpoints is needed for making composite string links via   the stacking product,
however, it is not convenient for the opposite process of decomposing a string link.
This is because, in the latter case, in general, the location of the endpoints of the parts obtained as a result of the decomposing process
is defined up to isotopy only.
Since in the present paper we focus on the task of decomposing a string link,
we will mainly use the braid equivalence.
Moreover, within the proof of the prime decomposition theorem (Theorem~\ref{theorem:Main})
in Section~\ref{sec:ProofOfMainTheorem}
we use the following weakened definition of a string link:

\begin{definition}
[Free string link]
\label{def:FreeStringLink}
By a \emph{free string link} we understand
a smooth embedding 
$I_1 \amalg \ldots \amalg I_n \to M \setminus (\partial \Sigma \times I)$
of $n$ pairwise disjoint copies of a segment into $M \setminus (\partial \Sigma \times I)$
so that for any $j=1,\ldots,n,$ the image of $I_j$
has one endpoint in $\partial_{-}$ and the other endpoint in $\partial_{+}$.
\end{definition}
Braid equivalence can be extended to free string links in the obvious way.

The decomposition process of a string link consists in consequently performing
cutting operations (cutting is defined in the next section)
and free string links naturally appear as a result of cutting.

\subsection{Cutting surfaces}
\label{sec:CuttingSurface}

\begin{definition}
[Cutting surface]
\label{def:CuttingSurface}
A properly embedded  surface $F \subset \Sigma \times I$
is called a \emph{cutting surface} for a string link $\ell \subset M$
if the following conditions hold:
\begin{itemize}
\item $F$ is isotopic to a (horizontal) fiber of the form $\Sigma \times \{ y \},y \in I$,
\item $F \cap \partial_{\pm} =\emptyset$,
\item $F$ intersects each strand of $\ell$ transversally in exactly one point.
\end{itemize}
\end{definition}
In particular, if $\partial \Sigma \not=\emptyset$
then $\partial F \subset \inter (\partial \Sigma \times I)$.
Note that a cutting surface is necessarily incompressible in $M$.

\begin{definition}
[$\ell$-admissible isotopy]
\label{def:AdmissibleIsotopy}
Let $F,F'$ be cutting surfaces for a string link $\ell$.
An isotopy taking $F$ to $F'$ is called \emph{$\ell$-admissible}
if at each moment of the transformation the surface remains a cutting surface for $\ell$.
\end{definition}

\begin{definition}
[Equivalent cutting surfaces]
\label{def:Equivalent CuttingSurfaces}
Cutting surfaces $F$ and $F'$ for a string link $\ell$ are \emph{equivalent}
if there is an $\ell$-admissible isotopy taking $F$ to $F'$.
In this case, we will write $F \simeq F'$.
\end{definition}

We denote by  the same symbol $\simeq$ both the braid equivalence of string links and  the equivalence of cutting surfaces.
This should not lead to confusion
because, in all cases, the type of object in question is clear from the context.

It is well-known that there is a correspondence between classical braids and isotopies of a disk.
An analogous correspondence for the same reasons takes place between surface braids and isotopies of the corresponding surface.
In our context, this implies that
disjoint cutting surfaces $F,F'$ for a string link $\ell$ are equivalent if and only if the free string link $\ell'$
that is the part of $\ell$ lying between $F$ and $F'$
is a braid up to braid equivalence.

A cutting surface $F$ cuts $M$ into two parts
each of which is homeomorphic to $\Sigma \times I$.
These parts will be denoted by $M_F^{+}$ and $M_F^{-}$,
where $\partial_{+} \subset M_F^{+}$ and $\partial_{-} \subset M_F^{-}$.
We define  both of these sets to be closed,
i.e., $M_F^{+} \cap M_F^{-} =F$.

A cutting surface $F$ cuts the string link $\ell$ into two free string links:
$\ell^{\pm} =\ell \cap M_F^{\pm}$.

\begin{definition}
[Trivial cutting surface]
\label{def:TrivialCuttingSurface}
A cutting surface $F$ for the string link $\ell$ is called \emph{trivial}
if at least one of $\ell^{+}$ and $\ell^{-}$
is a braid up to braid equivalence.
\end{definition}
Since any braid is braid equivalent to $\unit$
we can say that
a cutting surface is trivial if and only if
at least one of the parts into which it cuts the corresponding string link
is braid equivalent to $\unit$.

\begin{definition}
[Prime string link]
\label{def:Prime}
A non-braid string link is called \emph{prime}
if it has no non-trivial cutting surface.
\end{definition}
We do not consider braids as prime string links
although  they do not admit non-trivial cutting surfaces
(we discuss the fact below after Lemma~\ref{lemma:FC1}).

\subsection{Special pair of cutting surfaces}
\label{sec:SpecialPair}

Below in the proof of Theorem~\ref{theorem:Main}
we deal with pairs of cutting surfaces.
The proof becomes simpler if we consider
not arbitrary pairs but only pairs that are special in the following sense:

\begin{definition}
[Special pair of cutting surfaces]
\label{def:SpecialPair}
We will say that cutting surfaces $F_1,F_2$ form a \emph{special pair}
if the following conditions hold:
\begin{itemize}
\item if $\partial \Sigma \not=\emptyset$ then $\partial F_1 =\partial F_2$,
\item $\inter F_1$ and $\inter F_2$ are in general position
(i.e., they intersect transversely),
\item $F_1 \cap F_2$ either is empty 
or consists of a finite number of pairwise disjoint circles.
\end{itemize}
\end{definition}
Let  $F_1,F_2$ form a special pair of cutting surfaces.
Then the first condition implies that the equality $F_1 \cap F_2 =\emptyset$ can occur only if $\partial \Sigma =\emptyset$.
If $\partial \Sigma \not=\emptyset$ then $F_1$ and $F_2$ have coinciding non-empty boundary
consisting of pairwise disjoint circles.
The meaning of the last condition in the above definition is to exclude the situation when the intersection $F_1 \cap F_2$ contains arcs
with endpoints in $\partial F_j$.

The following lemma implies that within the decomposing process we can restrict ourselves
to cutting surfaces forming special pairs.

\begin{lemma}
\label{lemma:SpecialPair}
For any cutting surfaces $F_1,F_2$ for a string link $\ell$, there are
cutting surfaces $F'_1,F'_2$ so that
$F'_j \simeq F_j,j=1,2,$
and $F'_1,F'_2$ form a special pair.
\end{lemma}

\begin{proof}
Assume $F_1,F_2$ are already in general position.
If $\partial \Sigma =\emptyset$
then $F_1,F_2$ themselves form the desired pair.
Let $\partial \Sigma \not=\emptyset$.
Consider a circle  $\gamma \subset \partial \Sigma$ and the annulus $A =\gamma \times I \subset \partial M$.
Denote by $\gamma_j =F_j \cap A,j=1,2$.
By Definition~\ref{def:CuttingSurface} of a cutting surface,
 the surfaces $F_j$ are isotopic to a horizontal fiber,
hence the circles $\gamma_j$ are both isotopic to a circle of the form $\gamma \times \{ t \} \subset A$,
and thus they are isotopic to one another.
Choose circles $\gamma',\gamma'' \subset \inter \Sigma$ so that
\begin{itemize}
\item $\gamma, \gamma',\gamma''$ are pairwise isotopic pairwise disjoint and
\item $(B' \times I) \cap \ell =\emptyset$, where $B' \subset \Sigma$ denotes the annulus 
cobound by $\gamma$ and $\gamma'$ and
\item $\gamma'' \subset \inter B'$.
\end{itemize}
Such circles exist because, by Definition~\ref{def:StringLink} of a string link,
$\ell \subset \inter \Sigma \times I$.

We achieve the desired position of $F_2$ in two steps.
First, we isotope $F_2$ inside $B' \times I$ (keeping it fixed outside the set)
so that the circle $F_2 \cap A''$ lies above the circle $F_1 \cap A''$ in the annulus $A'' =\gamma'' \times I$ .
Second, we superpose circles $F_2 \cap A$ with $F_1 \cap A$
by an isotopy of $F_2$ that keeps $F_2$ fixed outside $B'' \times I$,
where $B'' \subset \Sigma$ is an annulus cobounded by $\gamma$ and $\gamma''$. 
The latter isotopy can be chosen so that
the third condition of Definition~\ref{def:SpecialPair} holds.
The resulting isotopy is $\ell$-admissible
because its support is contained in the set $B' \times I$
which has empty intersection with the string link $\ell$.
The surface $F_1$ remained fixed during the above process,
hence the pair consisting of $F_2$ in its new position and $F_1$
is the pair we need.
\end{proof}

\section{Main theorem}
\label{sec:MainTheorem}

Recall that by  $\Sigma$ we denote a compact connected orientable (not necessarily closed) surface other than $S^2$,
and $M =\Sigma \times I$.

\begin{theorem}
\label{theorem:Main}
A non-braid string link $\ell \subset M$
can be written as the product (under stacking) of prime factors 
\begin{equation}
\label{eq:Decomposition}
\ell =\ell_1 \# \ldots \# \ell_k, \quad k \geq 1,
\end{equation}
where this decomposition is unique up to possibly permuting the order of the factors in the right-hand side
and the prime factors are defined up to braid equivalence.
\end{theorem}

We would like to make the following comments concerning the theorem.

1. We do not consider braids as prime string links (see Definition~\ref{def:Prime}),
hence all terms in the right-hand side of the representation~\eqref{eq:Decomposition}
are not braids.

2. We are forced to say that the order of prime factors in the decomposition~\eqref{eq:Decomposition}
is not defined
because we do not know the structure of the center of the string link monoid.
In the case of classical $2$-component string links, the question is answered
in~\cite{BlairBurkeKoytcheff}.
The second situation when the structure of the center is known
is the semigroup of knotoids~\cite{Turaev2012},
which, in our terms, is the monoid of $1$-component string links in the thickened annulus.
But, as far as we know, the question is open both for classical string links with more than $2$ strands
and for string links  with more than $1$ strand in a thickened surface.
The only thing we can say in the general case
is that a local knot (we mean a knotted arc lying inside a sphere intersecting the link transversally in exactly $2$ points)
can be isotoped along the corresponding strand to any position.
Using this obvious fact we can construct string links
that commute with any other string link,
and the question is whether, in the general case, the center of the string link monoid contains an element of some other nature.
In the case of classical $2$-component string links, such elements exist(see~\cite{BlairBurkeKoytcheff}),
while, in the case of knotoids, by~\cite[Theorem 4.2]{Turaev2012}, do not.
We conjecture that, in the general case, the answer to the above question is negative.
More specifically,  the center of the monoid of string links in a thickened surface other than a disk or $2$-sphere
conjecturally has the structure  analogous to the one in the case  of the semigroup of knotoids.

3. String links in $S^2 \times I$ are excluded from Theorem~\ref{theorem:Main}.
The prime decomposition in this case should be studied individually
because such string links have some specific properties.
For example, any $1$-component string link in thickened sphere
is trivial because, in this case, any arc
with endpoints in distinct boundary spheres can be unknotted by an isotopy.
This leads to the following property:
any strand of any string link in thickened sphere can be unknoted
at the expense of an additional complicating of all other strands.
We think that this fact may affect
the uniqueness part of the corresponding prime decomposition theorem.

\section{The Diamond lemma}
\label{sec:DiamondLemma}

The lemma was proved by 
M.~H.~A.~Newman in 1942~\cite{Newman42}.
Later, it was called `the Diamond lemma'.
We will use it
in the form proposed by S.~V.~Matveev~\cite{MatveevRoots2012}.
Here we briefly review  Matveev's version of the Diamond lemma,
referring the reader to~\cite{MatveevRoots2012}
for more details of the approach.

Let $\Gamma$ be a directed graph.
We denote its vertex and edge sets by $\vGamma$ and $\eGamma$, respectively.
An \emph{oriented path} 
in $\Gamma$ is a sequence of edges
so that the end of each edge coincides with the beginning of the next one. 

We say that 
a vertex $w \in \vGamma$ 
is a \emph{root}  of a vertex  $v \in \vGamma$ 
if 
\begin{itemize}
\item there is an oriented path in $\Gamma$ from $v$ to $w$ and 
\item $w$ is a sink, that is, it does not have outgoing edges. 
\end{itemize}
In general, a vertex in a directed graph may have no roots, one root, or many roots.
The Diamond lemma answers the
question: 
under what conditions does each vertex of the graph $\Gamma$ have exactly one root? 
The lemma proposes two properties of the graph $\Gamma$.
Taken together, these properties are sufficient for a positive answer to this  question. 

The first property is called the 
\emph{finite path} property 
and  is denoted by (FP).

\emph{
(FP): 
Any oriented path in $\Gamma$ is finite.
}
In other words, $\Gamma$ does not contain oriented cycles and infinite oriented paths. 
The property (FP) implies that 
each vertex $v \in \vGamma$ 
has at least one root. 

Let the set $\eTwoGamma \subset \eGamma \times \eGamma$ 
be the set of pairs $(e_1,e_2),e_1,e_2 \in \eGamma,e_1 \not=e_2$,
so that the edges have the same beginning and distinct ends.

We now formulate the second property, 
which is denoted by (MF), 
referring to the `Mediator Function'. 

\emph{
(MF): 
There exists a map $\mu: \eTwoGamma \to \mathbb{N} \cup \{ 0 \}$, 
called the mediator function,
which satisfies the following conditions: 
}
\begin{itemize}
\item (MF1)
if $\mu(\overrightarrow{v_0v_1},\overrightarrow{v_0v_2}) =0$, 
then there exist oriented paths from $v_1$ and from $v_2$ 
ending at the same vertex $v_3 \in \vGamma$. 

\item (MF2)
if $\mu(\overrightarrow{v_0v_1},\overrightarrow{v_0v_2})  >0$ , 
then there exists an edge $\overrightarrow{v_0v_3} \in \eGamma$ 
so that $\mu(\overrightarrow{v_0v_j},\overrightarrow{v_0v_3}) < \mu(\overrightarrow{v_0v_1},\overrightarrow{v_0v_2})$
for $j=1,2$.
\end{itemize}

\noindent {\bf Lemma}
({\bf The Diamond lemma~\cite{MatveevRoots2012}}).
\emph{
If a directed graph $\Gamma$ has the properties (FP) and (MF), 
then each of its vertices has a unique root. 
}
\bigskip

Matveev suggested  the following scheme for proving typical prime decomposition theorems using the Diamond lemma.
Vertices of the graph $\Gamma$ are defined as finite collections of objects under consideration.
In our case, those are collections of string links.
Two vertices are connected with an edge if one can be obtained from the other as a result of a single simplification.
In our case, the simplification is cutting.
The oriented path represents a sequence of simplifications applying to the collection corresponding to its beginning.
The end of a path (the root) corresponds to a collection of indecomposable (prime) objects.
The mediator function, following Matveev's idea,  is a measure of proximity of the transformations corresponding to edges with the same beginning.
In our case, it is determined via intersection of cutting surfaces
(see Section~\ref{sec:MFDefinition} for details).
Therefore, the existence and uniqueness of the root of an arbitrary vertex
are a reformulation of the corresponding prime decomposition theorem.
Below, we implement this universal  scheme in the case of string links.

\section{Proof of Theorem~\ref{theorem:Main} }
\label{sec:ProofOfMainTheorem}

In this section we prove Theorem~\ref{theorem:Main}
using the Diamond lemma.
We start with the definition of the graph $\Gamma$ in Section~\ref{sec:GraphGamma}.
Then, we verify the conditions of the Diamond lemma:
the condition (FP) in Section~\ref{sec:FC}
and the existence of the mediator function in Section~\ref{sec:MF}.
The latter section is divided into four parts:
the definition of the function $\mu$ --- Section~\ref{sec:MFDefinition},
the verification of the condition (MF1) --- Section~\ref{sec:mf_1},
a technical lemma needed for verification of the condition (MF2) --- Section~\ref{sec:TechnicalLemma}
and the verification itself --- Section~\ref{sec:mf_2}.

All string links below, unless otherwise stated,
are $n$-component string links in the manifold $M =\Sigma \times I$.

\subsection{The graph $\Gamma$.}
\label{sec:GraphGamma}

Consider the set whose elements are
finite (non-empty) unordered collections of string links
\[
v =\{ \ell_1,\ldots,\ell_m \}, \quad m \geq 1,
\]
so that none of $\ell_1,\ldots,\ell_m$ is a braid.
Collections consisting of different numbers of string links are allowed.

We define an equivalence relation on the set:
$v \cong v'$ if and only if
there is a bijection $f: v \to v'$
so that $f(\ell) \simeq \ell$ for any $\ell \in v$.

We define \emph{vertices of the graph $\Gamma$}
to be the above collections of string links regarded up to the equivalence relation $\cong$.
The set of vertices of the graph $\Gamma$ will be denoted by $\vGamma$.
By abuse of notation, we will use the same notation both for a collection of string links and for the corresponding equivalence class.

Distinct vertices $v,v' \in \vGamma, v \not\cong v',$ are connected with a directed edge
$\overrightarrow{v,v'}$
if and only if there are two  members of $v'$ that  can be obtained from a member of $v$ as a result of a cutting
while all other members are common for the collections.
More precisely.
Let $v =\{ \ell_1,\ldots,\ell_m \}$,
$v' =\{ \ell'_1,\ldots,\ell'_{m+1} \}$
and one of the string links $\ell_1,\ldots,\ell_m$ (say, $\ell_1$)
can be cut  into string links $\ell_{1,1}, \ell_{1,2}$
so that
$\{ \ell_{1,1},\ell_{1,2},\ell_2,\ldots,\ell_m \} \cong \{ \ell'_1,\ldots,\ell'_{m+1} \}$.
The set of edges of the graph $\Gamma$
will be denoted by $\eGamma$.

It is necessary to emphasize the following: the edge $\overrightarrow{v,v'}$ implies that
$v$ can be transformed into $v'$ by a cutting,
but the vertices are connected with a single edge  independently on
whether there is a unique possibility to transform $v$ into $v'$ or not.
Therefore, a non-trivial cutting surface for a string link in a collection determines an edge of the graph $\Gamma$
while the converse is not true both for individual cutting surfaces
and for equivalence classes of cutting surfaces, as well.
For example, consider a collection $\{ \ell \# \ell, \ell \# \ell \}$ where $\ell$ is not a braid.
Then cutting surfaces that cut in half the first and the second string links in the collection
are inequivalent because they lie in distinct string links
but the corresponding cuttings give equivalent collections $\{\ell, \ell, \ell \# \ell \}$ and $\{ \ell \# \ell, \ell, \ell \}$.
One more example: cutting surfaces cutting off the first and last parts from $\{ \ell \# \ell \# \ell \}$ are inequivalent
but they give equivalent collections $\{ \ell, \ell \# \ell \}$ and $\{ \ell \# \ell, \ell \}$.

\subsection{Verification of the condition (FP)}
\label{sec:FC}

The idea we use below to check the finiteness condition
is similar to the one used in~\cite{BlairBurkeKoytcheff}
for proving the existence of the prime decomposition
in the case of classical $2$-component string links.

Fix a vertex $v =\{ \ell_1,\ldots,\ell_m \} \in \vGamma$.
The condition (FP) for the graph $\Gamma$ means that
any sequence of non-trivial cuttings that can be applied one by one to a collection of string links is finite.
Note that  it is sufficient to consider the case of a collection consisting of exactly one string link.
A sequence of sequential non-trivial cuttings of a string link is determined by the ordered collection of non-trivial cutting surfaces,
and all these surfaces can be embedded into the starting string link.
Therefore, to verify the condition (FP)
it is sufficient to prove the following two propositions.

\begin{lemma}
\label{lemma:FC1}
Let $S_1,S_2,S_3$ be pairwise disjoint cutting surfaces for a string link
so that $S_1 \simeq S_2$
and $S_3$ lies between $S_1$ and $S_2$.
Then $S_3 \simeq S_j$ for $j=1,2$.
\end{lemma}
In other words, a braid cannot be cut into two non-braid string links.
An analogous proposition for string links in the cylinder $D^2 \times I$ 
was proved originally in~\cite{Krebes}
and then proved using other arguments in~\cite{BlairBurkeKoytcheff}.
The latter proof does not use a specificity of the cylinder$D^2 \times I$,
hence it can be applied to our notion of string link without any changes.

In our context, the above lemma implies that
all surfaces involved in the collection of cutting surfaces representing a sequence of  non-trivial cuttings of a string link
are necessarily pairwise inequivalent.

\begin{lemma}
\label{lemma:FC2}
Any collection of pairwise inequivalent cutting surfaces
for a string link is finite.
\end{lemma}

The lemma is a consequence of the following theorem
proved in~\cite{FreedmanFreedman}.

\bigskip
{\bf Theorem~\cite{FreedmanFreedman}}.
\emph{
Let $\tilde{M}$ be a compact 3-manifold with boundary
and $b$ an integer greater than
zero. There is a constant $c(\tilde{M}, b)$ so that if $F_1,\ldots, F_k$, $k > c$, is a
collection of incompressible
surfaces such that all the Betti numbers $b_{1}F_{i} < b$, $1 \leq i \leq k$, and no
$F_i$, $1 \leq i \leq k$, is a boundary
parallel annulus or a boundary parallel disk, then at least two members $F_i$ and $F_j$
are parallel.
}
\bigskip

To prove Lemma~\ref{lemma:FC2}
using the theorem,
we transform $M =\Sigma \times I$ into $\tilde{M}$ by
removing from $M$ an open tubular  regular neighborhood of all components of the string link under consideration.
As a result of the transformation of the manifold, a cutting surface transforms into a surface in $\tilde{M}$,
which is homeomorphic to the surface before the transformation with $n$ additional holes
(here $n$ is the number of components of the string link).
The surface is incompressible in $\tilde{M}$.
Indeed, if $\Sigma$ is not a disk this follows from Definition~\ref{def:CuttingSurface};
arguments concerning the case of a disk can be found in~\cite{BlairBurkeKoytcheff}.
Equivalent cutting surfaces transform into parallel surfaces in $\tilde{M}$.
A trivial cutting surface becomes boundary parallel.
As the value of the constant $b$ we can use $b =b_1\Sigma +n+1$.

\subsection{Verification of the existence of the mediator function}
\label{sec:MF}

\subsubsection{Definition of the mediator function}
\label{sec:MFDefinition}

Let the mediator function $\mu: \eTwoGamma \to \mathbb{N} \cup \{ 0 \}$
is given by
\begin{equation*}
\label{eq:mu}
\mu(e_1,e_2) =\min_{(S_1,S_2)} |\inter S_1 \cap \inter S_2|
\end{equation*}
where the minimum is taken over all special pairs $S_1,S_2$ of  cutting surfaces
(see Definition~\ref{def:SpecialPair})
so that the edge $e_j \in \eGamma$ is determined by the surface $S_j,j=1,2$,
and $|\inter S_1 \cap \inter S_2|$ denotes the number of connected components of $\inter S_1 \cap \inter S_2$;
recall that, by Definition~\ref{def:SpecialPair},
the intersection consists of circles only.

As we mentioned above, the edges of the graph $\Gamma$
are determined by cutting surfaces.
Hence elements of $\eTwoGamma$ are determined by pairs of cutting surfaces,
and Lemma~\ref{lemma:SpecialPair}
guarantees that
we do not lose anything using only special pairs of cutting surfaces.
More precisely,
if cutting surfaces $S_1,S_2$ determine a pair $(e_1,e_2) \in \eTwoGamma$
then there is a special pair $S'_1,S'_2$ of cutting surfaces
so that $S'_j \simeq S_j,j=1,2,$
thus the pair $(S'_1,S'_2)$ determines the same pair $(e_1,e_2)$.

\subsubsection{Verification of the condition (MF1)}
\label{sec:mf_1}

Consider distinct edges
$e_1 =\overrightarrow{v_0,v_1}, e_2 =\overrightarrow{v_0,v_2} \in \eGamma$
so that $\mu(e_1,e_2) =0$.
This means that there are cutting surfaces $S_1,S_2$
having disjoint interiors,
so that cutting by $S_j$ transforms $v_0$ into $v_j,j=1,2$.
(Recall that, by definition, there are no multiple edges in the graph $\Gamma$, hence $v_1$ and $v_2$ are distinct vertices.)

We check that in this case there are a vertex $v_3 \in \vGamma$ and edges
$\overrightarrow{v_1,v_3}, \overrightarrow{v_2,v_3} \in \eGamma$.

Assume the surfaces $S_1,S_2$ lie in distinct string links in the collection
$v_0 =\{ \ell_1,\ell_2,\ldots,\ell_m \}$ (say, in $\ell_1, \ell_2$, respectively).
The cutting by $S_1$ transforms the collection $v_0$ into  $v_1 =\{ \ell_{1,1}, \ell_{1,2},\ell_2,\ldots,\ell_m \}$.
The transformation keeps the link $\ell_2$ unchanged, hence the cutting by $S_2$ can be performed.
As a result, we obtain $v_3 =\{ \ell_{1,1},\ell_{1,2},\ell_{2,1},\ell_{2,2},\ldots,\ell_m \}$.
Analogously, we can perform the same cuttings in opposite order
and obtain the same collection $v_3$, this time, from $v_2 =\{ \ell_1,\ell_{2,1},\ell_{2,2},\ldots,\ell_m \}$.

Assume $S_1,S_2$ lie in the same string link (say, $\ell_1$) in the collection $v_0$.
Since $\inter S_1 \cap \inter S_2 =\emptyset$,
the union $S_1 \cup S_2$ cuts $\ell_1$ into three free string links: $\ell_{1,1},\ell_{1,2},\ell_{1,3}$.
Let $S_1$ and $S_2$ separate $\ell_{1,1},\ell_{1,2}$
and $\ell_{1,2},\ell_{1,3}$, respectively.
Note that these three free string links are not braids.
Indeed, $\ell_{1,1}$ and $\ell_{1,3}$ are not braids because otherwise
$S_1$ and $S_2$ are trivial cutting surfaces.
The middle part $\ell_{1,2}$ is not braid because otherwise $S_1 \simeq S_2$
and hence $v_1 \cong v_2$.

Up to braid equivalence, the cutting by $S_1$ transforms $v_0$ into $v_2 =\{ \ell_{1,1}, \ell_{1,2} \# \ell_{1,3}, \ldots, \ell_m \}$.
After this, the surface $S_2$ remains non-trivial in $\ell_{1,2} \# \ell_{1,3}$ because it lies between $\ell_{1,2}$ and $\ell_{1,3}$,
which both are not braids.
Hence we can perform the cutting by $S_2$,
which gives $v_3 =\{ \ell_{1,1}, \ell_{1,2}, \ell_{1,3},\ldots,\ell_m \}$.
Analogous arguments show that the cuttings can be performed in the opposite order,
and the result will be again $v_3$, this time, obtained from
$v_2 =\{ \ell_{1,1} \# \ell_{1,2}, \ell_{1,3},\ldots,\ell_m \}$.

Therefore, in both of the above cases, we have two directed paths
from $v_0$ to $v_3$:
$\overrightarrow{v_0,v_1}, \overrightarrow{v_1,v_3}$
and $\overrightarrow{v_0,v_2},\overrightarrow{v_2,v_3}$.
This implies the condition (MF1).

\subsubsection{Main technical lemma}
\label{sec:TechnicalLemma}

Here we prove a technical lemma
required for the verification of the condition (MF2)
below in Section~\ref{sec:mf_2}.

Consider cutting surfaces (see Definition~\ref{def:CuttingSurface})
$S,T$ for a string link $\ell$ 
forming a special pair (see Definition~\ref{def:SpecialPair})
so that $\inter S \cap \inter T \not=\emptyset$.

Our assumptions  on $S$ and $T$ imply that
$S \cap M_T^{\pm} \not=\emptyset$.
Hence $T$ cuts $S$ into a finite number of connected regions $X_1,\ldots,X_m,m \geq 2$.
Let $\SByT$ denote the set 
\[
\SByT =\{ \bar{X_1},\ldots,\bar{X_m} \}
\]
where $\bar{\cdot}$ denotes  the closure of the specified set.
Analogously, we can define the set $\TByS$ 
consisting of closures of those parts into which $S$ cuts $T$.

Then for any $X \in \SByT$, the following holds:
\begin{itemize}
\item $X$ is connected,
\item $\inter X \not=\emptyset$ and $X \not=S$,
\item $\inter X \subset M_T^{+}$ or $\inter X \subset M_T^{-}$,
\item $\partial X \subset T$.
\end{itemize}
Any $Y \in \TByS$ has analogous properties.

Pick $X \in \SByT$.
Assume there is a closed subset $Y \subset T$
so that $\partial Y =\partial X$
(we do not assume here that $Y$ is connected or $Y \in \TByS$).
Then, we can perform the following transformation:
we cut off $Y$ from $T$ and then glue $X$ in its place.
The result is a surface,
which we denote by $U$:
\[
U =(T \setminus Y) \cup X.
\]
In this case, we will say
that $U$ is a \emph{rebuilding} of $T$ by $X$.

\begin{lemma}
\label{lemma:UIsCuttingSurface}
Let $S$ and $T$ be cutting surfaces for a string link $\ell$
so that
\begin{itemize}
\item they form a special pair and
\item $\inter S \cap \inter T \not=\emptyset$.
\end{itemize}
Then there are $X \in \SByT$ and $Y \in \TByS$
satisfying the following conditions:
\begin{itemize}
\item[\emph{(a)}] $\partial X =\partial Y$,
\item[\emph{(b)}] there is an isotopy fixing the common boundary of $X$ and $Y$
that takes $X$ to $Y$,
\item[\emph{(c)}] the union $X \cup Y$ is a closed orientable surface bounding a submanifold $\hat{XY} \subset M$,
\item[\emph{(d)}] the surfaces $U$ and $V$
that are rebuildings of $T$ by $X$ and $S$ by $Y$,respectively,
are cutting surfaces for $\ell$. 
\end{itemize}
\end{lemma}

\begin{proof}
Firstly, we prove that there are $X \in \SByT, Y \in \TByS$
so that (a)--(c) hold.
The main part of the proof is divided into three cases.
Lastly, we check that in all these cases, $U$ and $V$ are cutting surfaces for the link in question.

The proof below is divided into several  steps ---
italicized intermediate statements whose proofs  immediately  follow them. 

We begin with auxiliary terminology.
Let $F$ be a cutting surface for $\ell$
and $\gamma \subset F$ be an embedded oriented circle.
Consider a positive basis $(e_1,e_2,e_3)$ in $M$ so that
$e_1$ and $e_2$ are tangent vectors to $F$ at a point $x \in \gamma$
while $e_3$ is the normal vector to $F$ at $x$ oriented inward $M_F^{+}$.
We will say that a direction is \emph{to the left} of $\gamma$ in the surface $F$
if the direction coincides with the direction of $e_2$ when $e_1$ is a positive tangent vector to $\gamma$.
About the opposite direction, we will say that the direction
is the direction \emph{to the right} of $\gamma$ in the surface $F$.
Also, we will say that 
a closed region $Z \subset F$
\emph{lies to the left} of $\gamma$ if $\gamma \subset \partial Z$ and 
the direction to the left of $\gamma$ in $F$ 
is the direction inward $Z$.
Sometimes, if it is clear from the context what surface is meant by saying about the directions to the left/right of a circle,
the explicit specification of that surface will be omitted.

\emph{The rebuilding of $T$ by $X$ (resp. $S$ by $Y$) can be performed  for any $X \in \SByT$ (resp. $Y \in \TByS$).}

Since $S$ and $T$ form the special pair,
their intersection is a non-empty finite collection of pairwise disjoint circles,
hence
\[
\partial X =\gamma_1 \cup \ldots \cup \gamma_k,k \geq 1,
\]
where $\gamma_j,j=1,\ldots,k,$ are pairwise disjoint circles.
If $\partial S \not=\emptyset$, some of them, maybe, lie in $\partial S$.
Let the  circles be oriented so that
$X$ lies to the left of $\gamma_j$ in the surface $S$ for any $j=1,\ldots,k$. 

The circles $\gamma_1,\ldots,\gamma_k$ cobound the surface $X$ in $M$,
thus $[\gamma_1] +\ldots +[\gamma_k] =0$ in $H_1(M)$.
Since $M$ is homeomorphic to $T \times I$,
there is a natural isomorphism between the groups $H_1(T)$ and $H_1(M)$,
which maps the homology class of a loop in $T$ to the homology class of the same loop viewed as a loop in $M$.
Hence in $H_1(T)$ we also have $[\gamma_1] +\ldots +[\gamma_k] =0$.
Thus
$T \setminus \partial X$
is the union of two or more non-empty pairwise disjoint connected regions
(the number of regions is greater than $1$ because $X \not=S$).

Assume (aiming to reach a contradiction)
that there are two circles in $\partial X$
(say, $\gamma_1,\gamma_2$)
having the following properties:
\begin{itemize} 
\item There are points $x_j \in \gamma_j,j=1,2,$ and an arc $A =A[x_1,x_2] \subset T$
so that $A \cap \partial X =\{ x_1, x_2\}$.
\item The arc $A$ goes away from $\gamma_1$ to the left (in the surface $T$)
while it reaches $\gamma_2$ from the right.
\end{itemize}
In this case,
since $x_1,x_2 \in \partial X$ and $X$ is connected,
there is an arc 
$A' =A' [x_1, x_2] \subset X$.
We check that, under our assumptions, the loop $A \cup A'$ is an orientation reversing path in $M$.
Indeed, both arcs $A,A'$ go away from $\gamma_1$ to the left (in $T$ and $S$, respectively),
and $A'$ comes to $\gamma_2$ again from the left 
while $A$ comes to the same point in $\gamma_2$ from the right.
Since $S$ and $T$ are orientable, it means that narrow rectangles $R \subset T$ and $R' \subset X$
(having $A$ and $A'$ as their axial lines)
form the M\"obius strip.
This gives us the desired contradiction with the orientability of the manifold $M$
because surfaces $S,T$ are orientable and $\inter X$ is wholly contained either in $M_T^{+}$ or in $M_T^{-}$.

Therefore,there is $Y \subset T$
(which is not necessarily connected)
having the following properties analogous to those of $X$:
\begin{itemize}
\item $\inter Y \not=\emptyset$ and $Y \not=T$,
\item $\partial Y =\gamma_1 \cup \ldots \cup \gamma_k$,
\item $Y$ lies to the left of all circles in $\partial X =\partial Y$.
\end{itemize}

It is necessary to emphasize that the desired property $Y \in \TByS$ is not guaranteed for a while.
However, independently of whether $Y \in \TByS$ or not,
we can rebuild the surface $T$ by $X$
using the procedure described at the beginning of this section.
Analogous property of each region in $\TByS$ can be established by similar arguments.

\emph{The union $X \cup Y$ is a closed orientable surface that bounds a submanifold $\hat{XY} \subset M$.}

The union $X \cup Y$ is a closed surface.
The surface is orientable because:
\\1. $X \subset S,Y \subset T$ are subsurfaces of the orientable surfaces $S$ and $T$.
\\2. The orientations of $X$ and $Y$ induced by the orientation of $M$ using upward normal vectors of $S$ and $T$
are agreed in the following sense:
the upper side of $X$ is glued to the lower side of $Y$ and vice versa
in all circles where $X$ and $Y$ are glued.

The surface $X \cup Y$ does not separate $\partial_{-}$ and $\partial_{+}$.
For example, to go from $\partial_{-}$ to $\partial_{+}$ without intersecting $X \cup Y$,
we can intersect $S \setminus X$ and $T \setminus Y$
near any of $\gamma_j$ except for $\gamma_j \subset \partial S$
to the right of it
(such a circle exists because $X \not=S$).

Therefore, $X \cup Y$ bounds a submanifold $\hat{XY} \subset M$
so that
$\partial \hat{XY} =X \cup Y$.

Now we begin a proof that there are $X \in \SByT$ and $Y \in \TByS$
so that the properties (a)--(c) hold.

\begin{center}
{\sc Case 1:} there is a disk in $\SByT$.
\end{center}

The proof for the case $\TByS$ contains a disk is completely analogous.

Assume there is $X \in \SByT$
that is homeomorphic to a disk
(note $\partial X \not\subset \partial S$ because otherwise $X =S$).
Then the circle $\partial X$ bounds a disk $Y \subset T$
(otherwise the disk $X$ is a compressing disk for the surface $T$).
So we have the following fact, which will be used below.
\begin{remark}
\label{remark:gamma is not trivial}
If one of the collections $\SByT$ or $\TByS$ contains (resp. does not contain) a disk,
then the other one contains (resp. does not contain) a disk, too.
\end{remark}

The union $X \cup Y$ is a sphere.
We exclude the case $\Sigma =S^2$, and the manifold $M$ is irreducible,
hence the sphere $X \cup Y$ bounds a ball $\hat{XY}$,
 and we can isotope $X$ to $Y$ keeping the circle $\partial X =\partial Y$ fixed.
Since, by construction of rebuilding $T$ by $X$, the surfaces $S \setminus X$ and $T \setminus Y$ coincide,
this isotopy can be extended by identity to an isotopy taking $S$ to $T$.
The latter is isotopic to a fiber in the direct product $\Sigma \times I$,
hence $S$ is isotopic to a fiber, too.

\emph{There are $X$ and $Y$ as above, so that $X \in \SByT$ and $Y \in \TByS$.}

Let $X_1 \in \SByT$ and $Y_1 \subset T$,
be renamed disks as in the above discussion.
If $Y_1 \in \TByS$, we have the desired situation.
Assume $Y_1 \not\in \TByS$. It means $\inter Y_1 \cap S \not=\emptyset$.
The intersection consists of pairwise disjoint circles.
Each of them bounds a disk in $T$.
Choose in the collection of circles an innermost one,
and let $Y_2$ be a disk bounded by the circle.
Then $Y_2 \in \TByS$
and the above arguments, in which $S$ and $T$ are transposed, imply that
there is a disk $X_2 \subset S$ so that $\partial X_2 =\partial Y_2$.
If $X_2 \in \SByT$, we have the desired pair of disks.
If not, we can choose a disk $X_3 \subset X_2$ so that $X_3 \in \SByT$, and so on.
The process cannot be infinite 
because $S \cap T$ consists of a finite number of circles,
and each circle in the intersection bounds exactly one pair of disks in the surfaces under consideration,
hence by the way we cannot meet a circle more than once.

\begin{center}
{\sc Case 2:} $\SByT$ does not contain disks, but it contains an annulus.
\end{center}

Assume there is $X \in \SByT$
that is homeomorphic to an annulus.
Let $\partial X =\gamma_1 \cup \gamma_2$
where $\gamma_1,\gamma_2$ are boundary circles of $X$
oriented so that $X$ lies to the left of them.

Let $f =f(s,t): S^1 \times [0,1] \to X$ be an isotopy
taking $\gamma_1$ to $\gamma_2$ in $X$ (with orientation ignored),
and let $p_T: M \to T$ denote the projection map
coming from a representation $M =T \times I$ where $I$ is a segment.
Then the composition $p_t \circ f$
is a homotopy taking $\gamma_1$ to $\gamma_2$ in $T$.
It is known~\cite[Theorem 2.1]{Epstein66}
that two homotopic simple closed curves on an orientable   surface are isotopic.
Therefore, $\gamma_1,\gamma_2$ are isotopic in $T$.
Since $\gamma_1 \cap \gamma_2 =\emptyset$, the curves cobound an annulus $Y \subset T$.
Note the following fact for using below:
\begin{remark}
\label{remark:AnnuliInBothCollections}
If one of the collections $\SByT$ or $\TByS$ contains (resp. does not contain) an annulus,
then the other one contains (resp. does not contain) an annulus, too.
\end{remark}

As was shown above, the union $X \cup Y$ is an orientable surface.
In the case under consideration, it is a torus (not the Klein bottle),
and this torus bounds a submanifold $\hat{XY} \subset M$.

\emph{There is an isotopy fixing $\partial X =\partial Y$ that takes $X$ to $Y$.
}

To prove this fact we need to consider two cases
depending on whether the boundary circles of $Y$ separate the surface $T$ or not.

{\bf (a)}.
Suppose $\gamma_1$ (and hence $\gamma_2$ as well) does not separate the surface $T$.
In particular, this implies that both  $\gamma_1$ and $\gamma_2$ are not contained in $\partial T$.
Choose a simple closed curve $\delta \subset T$ so that
$\delta \cap \gamma_j =\{ z_j \},j=1,2$ where $z_j$ is a point.
Note that $\delta$ is non-trivial in $T$.
Consider the annulus $Z =p_T^{-1}(\delta)$.
Without loss of generality, we can assume that $Z$ and $X \cup Y$ are in general position.
Then, the intersection $Z \cap (X \cup Y)$ consists of circles,
which are
either trivial (bounding a disk) in the torus $X \cup Y$ or not.

All trivial circles can be eliminated one by one by an isotopy of the annulus $Z$ using the following standard procedure.
Consider an innermost (in the torus) circle $\alpha$.
It bounds a disk $D \subset (X \cup Y)$ so that
$D \cap Z =\alpha$.
The situation when $\alpha$ does not bound a disk in $Z$ is impossible
because, in this case, $\alpha$ is at the same time 
both trivial (because $\alpha =\partial D$)
and non-contractible (because it is isotopic to the non-trivial circle $\delta$).
Hence, $\alpha$ bounds a disk $D' \subset Z$.
The manifold $M$ is irreducible,
thus the sphere $D \cup D'$ bounds a ball in $M$.
This implies that there is an isotopy fixing $\alpha$ that takes $D'$ to $D$.
Then, we move $D'$ a bit further, and, as a result, the number of trivial circles in the intersection $Z \cap (X \cup Y)$ decreases by at least one.
After repeating this procedure (if needed), the intersection contains only circles that are non-trivial in the torus.

By construction, the intersection $Z \cap Y$ consists of only the arc $[z_1,z_2]$,
hence $Z \cap (X \cup Y)$ contains exactly one non-trivial circle.
The circle is composed of two arcs:
the first arc connects $z_1$ and $z_2$ in $X$, and the other one connects the same points in $Y$.
Since $X \cup Y$ bounds a submanifold,
the union of these arcs bounds a disk $D \subset Z$.
The disk is a compressing disk for $X \cup Y$.
As a result of the compression, we obtain a sphere.
 Since $M$ is irreducible, the sphere bounds a ball.
Hence, $\hat{XY}$ is a solid torus,
which can be represented as the direct product $D \times S^1$
so that $\gamma_1,\gamma_2$ are isotopic to a fiber,
and we can construct the desired isotopy using this structure.

{\bf (b)}.
Assume now that  $\gamma_j,j=1,2,$ separates the surface $T$.

{\bf (b1)}.
Let $\partial Y \cap \partial T =\emptyset$.
In this case, we need a curve $\delta \subset T$ having the following properties:
\begin{itemize}
\item $\delta$ intersects each of $\gamma_1, \gamma_2$ in exactly $2$ points:
$\delta \cap \gamma_1 =\{ z_1, z_3 \}, \delta \cap \gamma_2 =\{ z_2, z_4 \}$.
Let the points $z_1,\ldots,z_4$ be numbered so that, going along $\delta$ we meet  them in the order
$z_2,z_1,z_3,z_4$.
\item No part, into which $\delta \cup \gamma_1$ (or, equivalently, $\delta \cup \gamma_2$)
cuts the surface $T$, is homeomorphic to a disk.
\end{itemize}
Such a curve exists because, under our assumptions,  the curves $\gamma_j$ do not bound a disk in $T$
(see Remark~\ref{remark:gamma is not trivial}).

Assume, without loss of generality, that $X$ lies above $T$, i.e., $X \subset M_T^{+}$,
and let 
the annulus $Z =p_T^{-1}(\delta) \cap M_T^{+}$.
 Note that the curve $\delta$ is one of the boundary circles of the annulus.
Again, consider the intersection $Z \cap (X \cup Y)$
and, as above, eliminate all trivial circles in the intersection using an appropriate isotopy of the annulus $Z$.
As a result, the intersection consists of four arcs connecting the points $z_1,z_2,z_3,z_4$.
These are $a_{12} =a_{12}[z_1,z_2] \subset Y, a_{34} =a_{34}[z_3,z_4] \subset Y$
and two arcs lying in $X$.
If the latter are $b_{12} =b_{12}[z_1,z_2]$ and $b_{34} =b_{34}[z_3,z_4]$,
then we have a situation analogous to the one in the above case,
hence the desired isotopy exists.

The variant when the arcs are $b_{14} =b_{14}[z_1,z_4] \subset X$ and $b_{23} =b_{23}[z_2,z_3] \subset X$
is impossible because, in this case, these arcs  necessarily intersect each other inside $z$
(recall that going along $\delta$ we meet the points in the order $z_2,z_1,z_3,z_4$).

We check that the remaining option $b_{13} =b_{13}[z_1,z_3], b_{24} =b_{24}[z_2,z_4] \subset X$ is impossible, too.
Indeed, at least one of these arcs (say, $b_{13}$) cuts off  $Z$ a half-disk $D_1 \subset Z$ that does not contain the other arc.
Then $\partial D_1 =b_{13} \cup a_{13}$ where $a_{13} \subset \delta$.
Analogously, there is a disk $D_2 \subset X$
bounded by the circle $b_{13} \cup c_{13}$,
where $c_{13} =c_{13}[z_1,z_3] \subset \gamma_1$
(the points $z_1,z_3$ divide $\gamma_1$ into two arcs,
here we mean the arc that bounds a disk in $X$ together with $b_{13}$).
Hence there is a disk $D =D_1 \cup D_2$ so that
$\partial D =a_{13} \cup c_{13} \subset T$.
Since the surface $T$ is incompressible
its part bounded by $\partial D$ is homeomorphic to a disk.
Therefore, we obtained a contradiction because
  the part is one of the four parts into which $\delta \cup \gamma_1$ divides the surface $T$,
and we have chose the curve $\delta$ so that none of these four parts is homeomorphic to a disk.

{\bf (b2)}.
Assume now that $\partial Y \cap \partial T \not=\emptyset$.
The case $\partial Y \subset \partial T$ is impossible under our assumptions,
because  this would imply $T =Y$.
 Hence, a boundary circle of $Y$ (say, $\gamma_2$) is contained in $\partial T$,
while the other one is not.
The proof in this situation is analogous to the one in case (b1)
with the only difference being that $\delta$ is not a circle but an arc,
so that
\begin{itemize}
\item $\delta \cap \gamma_2 =\{ z_2,z_4\}$ (these are the endpoints of $\delta$) and $\delta \cap \gamma_1 =\{ z_1, z_3 \}$,
\item $\delta$ is not boundary parallel.
\end{itemize}
Such an arc exists because, otherwise, $T$ would be  a disk.

\emph{There are $X$ and $Y$ as above, so that $X \in \SByT$ and $Y \in \TByS$.}

Let $X_1 \in \SByT$ and $Y_1 \subset T$
be renamed annuli as in the above discussion.
If $Y_1 \in \TByS$ we have the desired situation.
Assume $Y_1 \not\in \TByS$. This means $\inter Y_1 \cap S \not=\emptyset$.
The intersection $\inter Y_1 \cap S$ consists of circles, which are non-trivial in $Y_1$
(recall that we assume that $S \cap T$ does not contain trivial circles).
Then $S \cap \inter \hat{X_1Y_1} \not=\emptyset$,
and each connected component of the intersection is an annulus whose boundary circles lie in $\inter Y_1$.
To see this, it is sufficient to apply arguments analogous to those we used above,
when we explained, why if $X$ is an annulus, then $Y$ is an annulus, too.
Now, in these arguments one should replace $X$ with the annulus lying inside $Y$ that is bounded by the boundary circles of the connected component of $S \cap \hat{X_1Y_1}$.
Choose any of the annuli in
$S \cap \inter \hat{X_1Y_1}$
 and denote it by $X_2$.
Then $X_2 \in \SByT$, and we can apply to $X_2$ the above arguments.
As a result, we have a pair of annuli $X_2 \in \SByT$ and $Y_2 \subset \inter Y_1$, analogous to the pair $X_1,Y_1$.
If $Y_2 \in \TByS$, we obtain the desired situation,
otherwise, we can choose the third pair of annuli, and so on.
The process is finite because  each time $Y_{i+1} \subset \inter Y_i$,
and the number of circles in $S \cap T$ is finite.

\begin{center}
{\sc Case 3:} $\SByT$ contains neither disks nor annuli.
\end{center}

Note that in this case (see Remarks~\ref{remark:gamma is not trivial} and~\ref{remark:AnnuliInBothCollections})
 $\TByS$ contains neither disks nor annuli, too.

Let $M_Y =p_T^{-1}(Y)$.
 Recall that by $p_T$ we denote  the projection  map $p_T: M \to T$
coming from a representation $M =T \times I$ where $I$ is a segment.

\emph{If $\SByT$ contains neither disks nor annuli,
then any $X \in \SByT$ can be isotoped
so that $X \subset M_Y$ and $X \cap \partial M_Y =\partial X$.
}

Fix $X \in \SByT$.
Denote by $A_j$
the annuli 
\[
A_j =p_T^{-1} (\gamma_j), \quad j=1,\ldots,k,
\]
where $\partial X =\gamma_1 \cup \ldots \cup \gamma_k =\partial Y$.

Fix $j, 1 \leq j \leq k$.
If $\gamma_j \subset \partial S$
then $X \cap A_j =\gamma_j$,
i.e., near $A_j$ we already have the desired position of $X$.

Assume $\gamma_j \not\subset \partial S$.
Let $X$ and $A_j$ be in general position and let
a small neighborhood of $\gamma_j$ in $X$ be inside $M_Y$
(if it is not so, this can be achieved by a small isotopy of $X$).
If $A_j \cap X =\gamma_j$, we have the desired position of $X$ near $A_j$.
Otherwise, the intersection consists of pairwise disjoint circles,
which can be of two types:
either trivial (i.e., bounding a disk in $A_j$),
or non-trivial (i.e., isotopic to $\gamma_j$, which is an axial line of $A_j$).

Assume $X \cap A_j$ contains a trivial circle.
Consider an innermost trivial circle $\delta \subset X \cap A_j$,
 and let $D \subset A_j$ be the disk bounded by $\delta$.
Then $D \cap X =\delta$.
The surface $S$ is incompressible,
hence $\delta$ bounds a disk $D' \subset S$.
If $D' \cap T \not=\emptyset$,
then the intersection contains a circle that bounds a disk in $S$,
but we have assumed above that $\SByT$ does not contain disks,
hence the situation is impossible under our assumptions.

Therefore, $D' \cap T =\emptyset$  or, equivalently, $D' \subset X$.
The manifold $M$ is irreducible,
hence the sphere $D \cup D'$ bounds a ball in $M$,
and we can isotope $D'$ to $D$ and then a bit further through $A_j$.
The isotopy of $X$ decreases the number of trivial circles in $X \cap A_j$ by at least $1$.
By repeating the transformation (if needed), we can eliminate all trivial circles in $X \cap A_j$.

Assume $X \cap A_j$ contains a non-trivial circle other than $\gamma_j$.
The number of such circles is necessarily even;
this is because:
\\1. Earlier we had isotoped a small neighborhood of $\gamma_j$ in $X$ to be inside $M_Y$,
hence a small neighborhood  of $\gamma_j$ in $A_j$ lies outside $\hat{XY}$.
\\2. If $X$ lies above (resp. below) $T$, a small neighborhood of the circle $A_j \cap \partial_{+}$ (resp. $A_j \cap \partial_{-}$) lies above (resp. below)$X$,
because $S \cap \partial M =\emptyset$;
hence the small neighborhood of the circle lies outside $\hat{XY}$.
\\3. The union $X \cup Y$ bounds a submanifold in $M$;
hence each circle in $X \cap A_j$ separates $A_j \cap \inter \hat{XY}$ and $A_j \cap \inter (M \setminus \hat{XY})$.
Thus $A_j$ is separated into alternating annuli lying inside and outside $\hat{XY}$.

Therefore, the number is even and $\geq 2$.

The circles are pairwise isotopic in $A_j$,
so any two of them cobound an annulus in $A_j$.
Consider non-trivial circles $c_1,c_2 \subset X \cap A_j$ other than $\gamma_j$,
which cobound an annulus $B \subset A_j$.
The circles $c_1,c_2$ are isotopic in $M$.
Let the corresponding isotopy be denoted by $f$.
Denote by $p_S$ the projection map $p_S: M \to S$
coming from a representation $M =S \times I$ where $I$ is a segment.
Then the composition $p_S \circ f$ is a homotopy taking $c_1$ to $c_2$ in the surface $S$.
Then, by~\cite{Epstein66},
$c_1$ is isotopic to $c_2$ in $S$.
Additionally, $c_1 \cap c_2 =\emptyset$,
hence $c_1,c_2$ cobound an annulus $B' \subset S$.

Let $c_1,c_2$ be so that
$B \cap X =c_1 \cup c_2$.
Such a pair of circles exists because $X \cup Y$ bounds a submanifold in $M$.
If $B' \cap T =\emptyset$,
then we can show that there is an isotopy fixing $\partial B =\partial B'$ that takes $B'$ to $B$.
This fact can be established using  the arguments we used above in case~$2$,
in which the surface $T$ and the annuli $X,Y$ should be replaced with the surface $S$ and the annuli $B',B$, respectively.
Hence we  can eliminate $c_1,c_2$ using an appropriate isotopy of $X$.
This decreases the number of non-trivial circles in $X \cap A_j$ by at least two.

Assume $B' \cap T \not=\emptyset$.
Then the intersection consists of circles
that are either trivial in $S$
or isotopic to the boundary circles of $B'$.
Both of these possibilities contradict our assumptions on $\SByT$,
because the first possibility implies $\SByT$ contains a disk,
and the second one implies $\SByT$ contains an annulus.

Therefore, we can eliminate all non-trivial circles in $X \cap A_j$ other than $\gamma_j$ using an appropriate isotopy of $X$.
As a result, we have $X \cap A_j =\gamma_j$.

Repeating the above procedure for all circles in $\partial X$, we isotope $X$ to a position so that
$X \cap \partial M_Y =\partial X$.
Additionally, we know that $X \cap \inter M_Y \not=\emptyset$
(e.g., points of $X$ lying near $\partial X$ were isotoped to be inside $M_Y$ at the beginning of the above process).
Hence $X \subset M_Y$.

In particular, this implies that $Y$ is connected
because, otherwise, $M_Y$ is disconnected while connected surface $X$ has non-empty intersection with each connected component of $M_Y$.

By construction, $Y \subset M_Y$.
Thus the submanifold $\hat{XY}$ bounded by $X \cup Y$ lies in $M_Y$.

\emph{There is an isotopy taking $X$ to $Y$
fixing their common boundary.
}

We will prove this using the following proposition proved by Waldhausen:

\bigbreak
\noindent {\bf Proposition}
(\cite[COROLLARY 5.5.]{Waldhausen68}).
\emph{
Let $Q$ be a compact connected orientable irreducible $3$-manifold.
Let $F$ and $G$ be incompressible 
orientable surfaces properly embedded in $Q$.
Suppose there is a homotopy from $F$ to $G$, 
which is constant on $\partial F$.
Then, $F$ is isotopic to $G$ by a deformation which 
is constant on $\partial Q$. 
}
\bigbreak

The only thing we need to show to use this statement 
is the existence  of the homotopy from $X$ to $Y$;
all other conditions for $M_Y,X,Y$ in the roles of $Q,F,G$, respectively, are satisfied by construction.
To this end, we consider two  projection maps
coming from a representation $M_Y =Y \times I$ where $I =[-1,1]$:
$p_Y: M_Y \to Y$ and $p_I: M_Y \to I$.
Let a map $h_t(x): M_Y \times [0,1] \to M_Y, t \in [0,1], x \in M_Y$,
be given by
\[
h_t(x) =y
\text{ where }
p_Y(y) =p_Y(x),
p_I(y) =(1 -t) p_I(x)
\]
(we think that $Y$ corresponds to the coordinate $0 \in [-1,1]$).
Since $X \subset M_Y$, the restriction of $h$ to $X$ is a homotopy
taking $X$ to a subset of $Y$.
But the subset coincides with $Y$
because $\partial X =\partial Y$ and $X$ separates $M_Y$.
To see the latter fact, it is sufficient to note that
 for any $z \in Y$, the vertical fiber $z \times I$ intersects $X$ at least once
because $X \cup Y$ bounds a submanifold in $M$.

Hence, by the aforementioned proposition, there is an isotopy taking $X$ to $Y$ fixing in their common boundary.
Extending the isotopy to $T \setminus Y$ by identity,
we obtain the desired isotopy taking $U$ to $T$.

\emph{There are $X$ and $Y$ as above, so that $X \in \SByT$ and $Y \in \TByS$.}

The proof in this case is completely analogous
to the one of the last italicized statement in the case 2,
with the only difference that now we are dealing with $X_i,Y_i$,
which are neither disks nor annuli.

\emph{The surfaces $U$ and $V$ are cutting surfaces for $\ell$.}

In all three of the above cases, we have proved
that the surface $U$ is isotopic to $T$,
which, by Definition~\ref{def:CuttingSurface} of a cutting surface,
is isotopic to a fiber in the direct product $\Sigma \times I$;
hence $U$ is isotopic to a fiber, too.

It remains to check that
$U$ intersects each component of the link $\ell$ exactly once.
This is a consequence of our construction of the surface $U$, that is, $X \cup (T \setminus Y)$,
and the fact, which we established above in all three of the above cases,
that the union $X \cup Y$ bounds a submanifold $\hat{XY} \subset M$
(it is a ball in the first case, a solid torus in the second one, and a handlebody of the genus $\geq 2$ in the third one).
Indeed, consider a component $K$ of the link $\ell$.
Since $X \subset S$ and $S$ is a cutting surface,
$K$ intersects $X$ in at most $1$ point.
The surface $T$ is a cutting surface, hence $K$ intersects $T$ exactly once.
If the intersection point lies in $\inter Y$,
then $K$ comes here into $\hat{XY}$,
and then it necessarily intersects $X$,
because it must exit $\hat{XY}$ to have an endpoint in $\partial M$.
If $K$ intersects $T$ in $T \setminus \inter Y$,
then $K$ does not intersect $\inter X$
because otherwise, it would come their into $\hat{XY}$, and hence necessarily have the second intersection point with $X \cup Y$, where it goes outward $\hat{XY}$.
\end{proof}

\subsubsection{Verification of the condition (MF2)}
\label{sec:mf_2}

Consider edges $e_1,e_2 \in \eGamma, e_1 =\overrightarrow{v_0,v_1}, e_2 =\overrightarrow{v_0,v_2},$
so that $\mu(e_1,e_2) >0$.
We need to show that there is an edge $e_3 =\overrightarrow{v_0,v_3} \in \eGamma$ so that
\begin{equation}
\label{eq:mu e_3}
\mu(e_1,e_3) < \mu(e_1,e_2) \text{ and } \mu(e_2,e_3) < \mu(e_1,e_2).
\end{equation}

Let $S,T$ be cutting surfaces
that determine the edges $e_1,e_2$, respectively,
and so that
\[
|\inter S \cap \inter T| =\mu(e_1,e_2) >0.
\]
In particular, this means that $S$ and $T$ are inequivalent cutting surfaces for the same string link $\ell \in v_0$,
and the surfaces satisfy the conditions of Lemma~\ref{lemma:UIsCuttingSurface}.
Hence, there is $X \in \SByT$ so that 
the surface $U$,
which is the rebuilding of $T$ by $X$
(the procedure giving the surface described at the beginning of Section~\ref{sec:TechnicalLemma}),
is a cutting surface for the string link $\ell$.

By construction, $U =(T \setminus Y) \cup X$.
Let $U'$ be the surface obtained as a result of
moving $U$ a bit upward if $X$ lies above $T$ or a bit downward otherwise
(keeping $\partial U$ fixed).
The resulting surface is equivalent to $U$.
Since $X \cap T =\partial X$, we have 
\begin{equation}
\label{eq:U less S}
|\inter U' \cap \inter T| =0 < |\inter S \cap \inter T|.
\end{equation}
By construction,
\begin{equation}
\label{eq:U less T}
|\inter U' \cap \inter S| \leq |\inter S \cap \inter T| -k,
\end{equation}
 where $k =|\partial X\cap \inter S| >0$
because all the circles in $\partial X$ that are not in $\partial S$ disappear as a result of rebuilding
and $\partial X \not\subset \partial S$.
The number $k$ is greater than $0$ because $X \not=S$.

Therefore, if $U$ is a non-trivial cutting surface then it determines an edge $e_3 \in \eGamma$, and the edge is what we need
(cf.~\ref{eq:U less S},\ref{eq:U less T} and~\ref{eq:mu e_3}).
Hence to complete the proof of Theorem~\ref{theorem:Main},
it is sufficient to show that in the situation under consideration,
there is a rebuilding that gives a non-trivial  cutting surface.
The fact is a consequence of the following lemma.

\begin{lemma}
\label{lemma:if U is trivial}
Let $S$ and $T$ be cutting surfaces for an $n$-strand string link $\ell$
having the following properties:
\begin{itemize}
\item[\emph{(a)}] $S$ and $T$ are non-trivial cutting surfaces for $\ell$,
\item[\emph{(b)}] $|\inter S \cap \inter T| > 0$,
\item[\emph{(c)}] all cutting surfaces that can be obtained as a result of rebuilding one of the surfaces by the other
(both $S$ by $T$ and $T$ by $S$)
are trivial cutting surfaces for the link $\ell$.
\end{itemize}
Then the (unordered) pair of string links obtained  as the result of cutting $\ell$ by $S$ 
coincide up to braid equivalence
with the pair obtained as a result of cutting $\ell$ by $T$.
\end{lemma}

We will prove the lemma below.
Let us now finish the proof of Theorem~\ref{theorem:Main}.
If there is $X \in \SByT$ or $Y \in \TByS$ so that
the corresponding rebuilding is a non-trivial cutting surface,
then, as we  have shown before the lemma,  the rebuilding determines the desired edge $e_3 \in \eGamma$.
Otherwise, if all rebuildings are trivial cutting surfaces,
then, by the lemma, cuttings by $S$ and $T$ give equivalent
(in the sense  of the equivalence relation $\cong$ defined in Section~\ref{sec:GraphGamma})
pairs of string links.
But this implies
that $v_1 \cong v_2$.
Therefore, we reach a contradiction with  the condition $e_1 \not=e_2$,
because, by the definition of the  graph $\Gamma$ (Section~\ref{sec:GraphGamma}),
edges having coinciding endpoints coincide.

\emph{Proof of Lemma~\ref{lemma:if U is trivial}.}
The surfaces $S$ and $T$ satisfy the conditions of 
Lemma~\ref{lemma:UIsCuttingSurface}.
Let $X \in \SByT$, $Y \in \TByS$, and
cutting surfaces $U$ and $V$
be those that exist by the lemma.
Assume, without loss of generality, that $X$ lies above $T$.

By condition~(c), $U$ and $V$ are trivial cutting surfaces.
The surfaces $U$ and $T$ are isotopic,
but the isotopy cannot be $\ell$-admissible
(see Definition~\ref{def:AdmissibleIsotopy})
because otherwise $T$ is equivalent to the trivial cutting surface $U$, and thus is itself trivial, contradicting condition~(a).
Therefore, $\ell_{XY} =\ell \cap \inter \hat{XY}$
after an appropriate isotopy can be viewed as  a non-trivial (free) string link in  $Y \times I$ where $I$ is a segment.
It is necessary to emphasize that $\ell_{XY}$ is defined up to braid equivalence.

Let $C_1,\ldots,C_r,1 \leq r \leq n,$ be strands of the link $\ell$
so that $C_i \cap \inter \hat{XY} \not=\emptyset$ if and only if $1 \leq i \leq r$.
The strands form a sublink of $\ell$,
we denote it by $\ell_1$.
Then $\ell_{XY}$ is the part of $\ell_1$ lying between $X$ and $Y$.
We denote by $\ell_{1a}$ and $\ell_{1b}$
parts of $\ell_1$ lying above $X$ and below $Y$, respectively.
Since $U$ is a trivial cutting surface,
the part of $\ell$ that lies above $U$ is a braid,
and $\ell_{1a}$, which is its subset, is a braid, too.
Hence the part of $\ell$ lying above $T$ 
is (up to braid equivalence) the union
of $\ell_{XY} \# \ell_{1a} \simeq \ell_{XY}$
and  a  braid
$\ell_{2a}$, that is, the upper (lying above $T$) part of $\ell_2$,
where $\ell_2$ is the sublink of $\ell$ formed by 
the strands $C_{r+1},\ldots,C_n$.
Below, we will show that $\ell_2$  is not empty.

By the condition~(c), the surface $V$ is a trivial cutting surface,
hence it cuts off a braid from $\ell$, and the braid lies below $V$
(otherwise, the part of $\ell$ lying above $V$ is a braid,
but this is impossible because the part contains $\ell_{XY}$).
Thus $\ell_{1b}$, which is its subset, is also a braid.

Therefore, $\ell_1$ consists of $3$ parts:
$\ell_{XY}$ (in the middle) and two braids (from below and from above).
Hence $\ell_1 \simeq \ell_{XY}$.
In particular, this implies that $\ell_1$ is not a braid.

By the condition~(a), the surface $T$ cuts $\ell$ into two non-trivial string links.
Hence the part of $\ell$ lying below $T$ is non-trivial.
The part is the union of the braid $\ell_{1b}$
and the part of $\ell_2$ lying below $T$.
The latter is not a braid because otherwise $T$ is trivial cutting surface.
As we have mentioned above, the part of $\ell_2$ lying above $T$ is a braid,
hence $\ell_2$ is a non-empty non-trivial string link,
and the part of $\ell$ lying below $T$
is (up to braid equivalence) $\ell_2 \amalg \ell_{1b}$.
Therefore, $T$ cuts $\ell$ into
\[
\begin{split}
\ell_1 \amalg \ell_{2a} \simeq \ell_1 \amalg \unit_{c_{r+1},\ldots,c_n} =\ell' \quad \text{lying above $T$ and} \\
\ell_2 \amalg \ell_{1b} \simeq \ell_2 \amalg \unit_{c_1,\ldots,c_r} =\ell'' \quad \text{lying below $T$}.
\end{split}
\]

Applying analogous arguments to the surface $S$,
we can show that
up to braid equivalence $S$ cuts $\ell$ into
\[
\begin{split}
\ell_2 \amalg \ell_{1a} \simeq \ell_2 \amalg \unit_{c_1,\ldots,c_r} =\ell'' \quad \text{lying above $S$ and} \\
\ell_1 \amalg \ell_{2b} \simeq \ell_1 \amalg \unit_{c_{r+1},\ldots,c_n} =\ell' \quad \text{lying below $S$}
\end{split}
\]
where $\ell_{2b}$ is the part of $\ell_2$ lying below $S$.

Therefore, in the situation under consideration,
the surfaces $S$ and $T$ 
cut $\ell$ into $\ell'$ and $\ell''$.
The difference between these two cuttings consists in
which part lies above the corresponding cutting surface and which part lies below it,
but this difference does not matter in our context.
\\Lemma~\ref{lemma:if U is trivial}
is proved.
\qed

This completes the proof of Theorem~\ref{theorem:Main}.

\section*{Acknowledgments}
The author is very grateful to the reviewer for careful reading the paper, and for his helpful  comments and suggestions.

This work was supported by the Russian Science Foundation grant~22-11-00299.

\bibliographystyle{plain}
\bibliography{CuttingOfKnotsArxiv}

\end{document}